\documentclass[twoside,leqno,10pt]{amsart}
\usepackage{amsfonts}
\usepackage{amsmath}
\usepackage{amscd}
\usepackage{amssymb}
\usepackage{amsthm}
\usepackage{latexsym}
\usepackage{color}
\newtheorem{theorem}{Theorem}

\newtheorem{Proposition}{Proposition}

\numberwithin{equation}{section}

\title{equidistribution of solutions of ternary quadratic congruences MODULO PRIME POWERS}
\author{Anup Haldar} 
\address{Anup Haldar,
Ramakrishna Mission Vivekananda Educational and Research Institute, Department of Mathematics, G. T. Road, PO Belur Math, Howrah, West Bengal 711202, India}
\email{anuphaldar1996@gmail.com}

\begin{document}
\begin{abstract}
Let $p$ be a fixed odd prime and $Q(x,y,z)=ax^2+bxy+cy^2+dxz+eyz+fz^2$ be a fixed quadratic form in $\mathbb{Z}[x,y,z]$ which is non-degenerate in $\mathbb{F}_p[x,y,z]$ and $(a(4ac-b^2),p)=1.$ Let $(x_0,y_0,z_0)$ be a fixed point in $\mathbb{Z}^3$. We study the behavior of solutions $(x,y,z)$ of congruences of the form $Q(x,y,z)\equiv0\bmod{q}$ with $q=p^n,$ where max$\{|x-x_0|,|y-y_0|,|z-z_0|\}\leq N$ and $(z,p)=1.$ In fact, we consider a smooth version of this problem and  establish an asymptotic formula (thus the existence of such solutions) when $n\rightarrow\infty$, under the condition $N\geq q^{\frac{1}{2}+\varepsilon}$.
\end{abstract}
\maketitle
\tableofcontents

\section{Introduction and main results} 
Recently, the author co-published an article \cite{arx} titled "Asymptotic behavior of small solutions of quadratic congruences in three variables modulo prime powers". In this work, we studied small solutions of diagonal quadratic congruences of the form 
\begin{equation}\label{prev}
    ax^2+cy^2+fz^2\equiv0\bmod{q},
\end{equation}
where $a,c,f\in\mathbb{Z}$, $q=p^n$  is a power of a fixed odd prime $p$ and $n\rightarrow\infty.$ More precisely, we investigated the asymptotic behavior of small solutions of the congruence \eqref{prev} with max$\{|x|,|y|,|z|\}\leq N$ and $(xyz,p)=1 .$ We also assumed $(acf,p)=1.$ For convenience we considered a smooth version of this problem (i.e. the solutions are suitably weighted) and studied the problem both for arbitrary and fixed coefficients $a,c,f.$ In the case of fixed coefficients $a,c,f$ and $n\rightarrow\infty$, we obtained an asymptotic formula (\cite[Theorem 1]{arx}) if $N\gg q^{1/2+\varepsilon}$, and in the case of coefficients which are allowed to vary with $n$, we obtained such a formula (\cite[Theorem 2]{arx} ) for $N\gg q^{11/18+\varepsilon}$. \\Let \begin{equation}
    Q(x,y,z)=ax^2+bxy+cy^2+dxz+eyz+fz^2
\end{equation} be a quadratic form in $\mathbb{Z}[x,y,z]$ and $\Delta:=\Delta_Q$ be the determinant of the associated matrix
$$
A_Q=\begin{pmatrix}
a & \frac{b}{2} & \frac{d}{2}\\
\frac{b}{2} & c & \frac{e}{2}\\
\frac{d}{2}&\frac{e}{2}&f
\end{pmatrix}.
$$
In this present paper, we study the asymptotic behavior of the solutions of the congruence 
\begin{equation}\label{eqq}
    Q(x,y,z)\equiv 0\bmod{q},
\end{equation}
with fixed coefficients and $(a(4ac-b^2)\Delta,p)=1$ (Note that $a$,$\frac{4ac-b^2}{4}$ and $\Delta$ are leading principal minors of the associated matrix $A_Q$). Here, for a fixed point $(x_0,y_0,z_0)\in\mathbb{Z}^3$, we consider the set of solutions $(x,y,z)$ with $(z,p)=1,$ and max$\{|x-x_0|,|y-y_0|,|z-z_0|\}\leq N$ obtaining a similar asymptotic formula as in \cite{arx} if $N\gg q^{1/2+\varepsilon}$. The condition $(z,p)=1$ is slightly different than the condition $(xyz,p)=1$ and it also excludes the trivial solution (0,0,0). The author is in a search of a suitable way to generalize \cite[Theorem 2]{arx} to general quadratic forms. Readers are requested to see \cite{arx} for a more detailed introduction.
To state our result, we define the following quantity
$$
C_p(Q):=\frac{(p-s_p(Q)))(p-1)}{p^2},
$$
where
\begin{equation} \label{sdef}
S_p(Q):=\left(\frac{b^2-4ac}{p}\right)
.\end{equation}
We note that if $q=p$ is an odd  prime, then the total number  of solutions $(x,y,z)$ to the congruence \eqref{eqq} satisfying $(z,p)=1$ turns out to be $(p-1)(p-S_p(Q))$.  So a solution exists as $p-s_p(Q)\geq 2$.  Our main result is as follows. 
 
\begin{theorem} \label{mainresult}
Let $\varepsilon>0$ be fixed, $p$ be a fixed odd prime, and $Q(x,y,z)=ax^2+bxy+cy^2+dxz+eyz+fz^2$ be a fixed  quadratic form in $\mathbb{Z}[x,y,z]$  with $(a(4ac-b^2)\Delta,p)=1.$ Let $\Phi:\mathbb{R}\rightarrow \mathbb{R}_{\ge 0}$ be a Schwartz class function. Set $q:=p^n$. Then as $n\rightarrow \infty$, we have the asymptotic formula
\begin{equation} \label{main}
\sum\limits_{\substack{(x,y,z)\in \mathbb{Z}^3\\ (z,p)=1\\ Q(x,y,z) \equiv 0 \bmod{q}}} \Phi\left(\frac{x-x_0}{N}\right)
\Phi\left(\frac{y-y_0}{N}\right)\Phi\left(\frac{z-z_0}{N}\right)\sim
\hat{\Phi}(0)^3\cdot C_p(Q)\cdot \frac{N^3}{q},
\end{equation}
provided that $N\ge q^{1/2+\varepsilon}$.
\end{theorem}
One can see that the right-hand side of \eqref{main} does not depend on the choice of the fixed point $(x_0,y_0,z_0)$, so the solutions $(x,y,z)$  of $Q(x,y,z) \equiv 0 \bmod{q}$, with $(z,p)=1$ are equally distributed throughout $\mathbb{Z}^3$. Our method is a generalization of the method that is used in the proof of \cite[theorem 1]{arx}. Key ingredients in our method are a parametrization of $\mathbb{Q}_p$- rational points $(x,y)$ on the conic  $$ax^2+bxy+cy^2+dx+ey+f=0,$$
repeated use of Poisson summation and an explicit evaluation of complete exponential sums with rational functions to prime power moduli due to Cochrane \cite{CoZ}. This transforms the problem into a dual problem which amounts to counting solutions of quadratic Diophantine {\it equations} (rather than {\it congruences}).
Then we use a result due to Heath-Brown\cite[corollary 2]{Hb} for counting the number of the solutions of the obtained dual quadratic equation.

{\bf Acknowledgements.} I would like to express my deep gratitude to Professor Stephan Baier, my research supervisor, for his patient guidance, enthusiastic encouragement, and useful critiques of this research work. Secondly, I would like to thank CSIR, Govt. of India for financial support in the form of a Senior Research Fellowship under file number 09/934(0016)/2019-EMR-I. \\ \\
{\bf Data availability statement:} This manuscript has no associated data.

\section{Preliminaries}
We will use the notation 
$$
e_q(z):=e\left(\frac{z}{q}\right)=e^{2\pi i z/q}
$$
for $q\in \mathbb{N}$ and denote by $G_q$ the quadratic Gauss sum
$$
G_q=\sum\limits_{x=1}^q e_q(x^2).
$$
We recall that if $q=p^k$ for some odd prime $p$ and $k\in\mathbb{N}$, then 
\begin{equation}\label{Fd}
    G_q=\begin{cases}
   p^{\frac{k}{2}}~~~&\text{if}~k~~\text{even,}\\
   p^{\frac{k-1}{2}}G_p~~&\text{if}~k~~\text{odd,}
\end{cases}
\end{equation}
where $\left(\frac{y}{q}\right)$ is the Jacobi symbol. We also denote by $q(x,y)$ the de-homogenized form of the quadratic form $Q(x,y,z),$ 
\begin{equation}\label{qxy}
    q(x,y)=ax^2+bxy+cy^2+dx+ey+f.
\end{equation}
Throughout the sequel, we will write $\overline{\alpha}$ for a multiplicative inverse of $\alpha$ to the relevant modulus, which will always be apparent from the context. 

The following preliminaries will be needed in the course of this paper.
\begin{Proposition}[Parametrization of points on a conic] \label{para}
Let $K$ be a field, $(\alpha,\beta)\in K^2$ and $ q(x,y)\in K[x,y]$ be as in \eqref{qxy}  with $a\Delta\neq 0$ such that 
$$
q(\alpha,\beta)=0,
$$
and $q_x(x,y),q_y(x,y)$ be the partial derivatives of $q(x,y)$ with respect to $x$ and $y$ respectively.
Then, the map 
\begin{equation} \label{bijection}
s : \left\{t\in K : at^2+bt+c\not=0\right\}\cup\{\infty\} \longrightarrow \left\{(x,y)\in K^2 :q(x,y)=0\right\}
\end{equation}
defined by 
\begin{equation} \label{mt}
s(t):=\begin{cases}
\left(\alpha-tM(t),\beta-M(t)\right)~~\text{if}~~t\in K,\\
\left(\alpha-\frac{A}{a},\beta\right)~~\text{if}~~t=\infty,
\end{cases}
\end{equation}
is bijective, where
$$M(t)=\frac{At+B}{at^2+bt+c},$$
with $A=q_x(\alpha,\beta)$ and $B=q_y(\alpha,\beta).$ 
\end{Proposition}

\begin{proof}
We use a standard method of parametrization. Given a $K$-rational point $P=(\alpha,\beta)$ on the conic and $t\in K\cup\{\infty\}$, by B\'ezout's theorem, the line $\mathcal{L}(t)$ through $P$ given by the equation $x-\alpha=t(y-\beta)$ intersects the conic in $P$ and at most one more point $s(t)$(where by $\mathcal{L}(\infty)$ we mean the line $y=\beta$). This point may be $P$ itself, in which case $\mathcal{L}(t)$ is the tangent to the conic at $P$. Conversely, if $Q$ is a $K$-rational point on the conic, then there exists precisely one line through $P$ and $Q$ (which is the tangent in the case $P=Q$, and the tangent is unique because $\Delta\neq 0$) with rational slope $t$ or $\infty$. Hence, we have a bijection between the $K$-rational points on the conic and the set of $t\in K\cup\{\infty\}$ for which $s(t)$ exists. To find $s(t)$, we write the Taylor series expansion around $(\alpha,\beta)$ in the form
\begin{align*}
q(x,y)=&q(\alpha,\beta)+q_x(\alpha,\beta)(x-\alpha)+q_y(\alpha,\beta)(y-\beta)\\
&\ \ \ \ +\frac{1}{2}q_{xx}(\alpha,\beta)(x-\alpha)^2+q_{xy}(\alpha,\beta)(x-\alpha)(y-\beta)+\frac{1}{2}q_{yy}(\alpha,\beta)(y-\beta)^2.
\end{align*}
Now we plug in $ x-\alpha=t(y-\beta)$ and get
$$(At+B)(y-\beta)+(at^2+bt+c)(y-\beta)^2=0.$$
Therefore, if $y\neq\beta$ (i.e when $t\in K$)
\begin{equation*}
\begin{split}
y=&\beta-\frac{At+B}{at^2+bt+c}
\end{split}
\end{equation*}
and 
\begin{equation*}
    \begin{split}
        x&=\alpha+(y-\beta)t\\
        &=\alpha-t\frac{At+B}{at^2+bt+c},
    \end{split}
\end{equation*}
which exist if $at^2+bt+c\neq 0.$\\
If $y=\beta$ (i.e when $t=\infty$), we have
$$ax^2+b\beta x+c\beta^2+dx+e\beta+f=0,$$
which is a quadratic polynomial in $x$ (as $a\not=0$) having two roots $x=\alpha $ and $\alpha-\frac{A}{a}.$
This gives the desired parametrization in \eqref{mt}, and the map in \eqref{bijection} is bijective.
\end{proof}

\begin{Proposition} \label{nofroot}
Let $p>2$ be a prime, $n\in \mathbb{N}$ and $ q(x,y)\in\mathbb{Z}[x,y]$ be as in \eqref{qxy}  with $(a\Delta,p)=1$. Then the number of solutions $(x,y) \bmod p^n$ of the congruence 
$$
q(x,y)\equiv 0\bmod{p^n}
$$ 
equals
  \[N_{p^n}(q) =
    p^{n-1}\left(p-\left(\frac{b^2-4ac}{p}\right)\right).
  \]
\end{Proposition}

\begin{proof}
Since $a\not\equiv0\bmod{p}$, the number of solutions of the congruence $at^2+bt+c\equiv0\bmod{p}$ is $1+\left(\frac{b^2-4ac}{p}\right)$.
So using Proposition \ref{para} with $K=\mathbb{F}_p$, we get that the number of solutions of the congruence 
           $$
           q(x,y)\equiv0\bmod{p}
           $$
           is $N_p(q)=p-\left(\frac{b^2-4ac}{p}\right)$. Now, we use a Hensel-type argument as follows. Let $(x,y)$ be a solution of the congruence 
           \begin{equation}
               q(x,y)\equiv0\bmod{p^k}.
           \end{equation}
           We want to lift it to a solution $(\tilde{x},\tilde{y})$ of the congruence 
           \begin{equation}
                q(x,y)\equiv0\bmod{p^{k+1}}.
           \end{equation}
           So we write $\tilde{x}=x+k_1p^k,\tilde{y}=y+k_2p^k$ and use the Taylor series expansion around $(x,y)$ to get
           $$
           \frac{q(x,y)}{p^k}+q_x(x,y)k_1+q_y(x,y)k_2\equiv0\bmod{p}.
           $$
           Now  $(\Delta,p)=1$ implies $(q_x(x,y),q_y(x,y))\not\equiv(0,0)\bmod{p}$. So the congruence in $k_1,k_2$ has exactly $p$ solutions.
           Thus, 
               $$
               N_{p^{k+1}}(q)=pN_{p^{k}}(q).
               $$
               Hence, using mathematical induction the Proposition is proved.
\end{proof}
\begin{Proposition}\label{Par}
Let $p>2$ be a prime, $n\in \mathbb{N}$ and $ q(x,y)\in \mathbb{Z}[x,y]$ be as in \eqref{qxy} with $(a\Delta,p)=1$. Assume that $q(\alpha,\beta)\equiv0\bmod{p^n}.$ Then the solutions $(x,y) \bmod p^n$ of the congruence 
$$
q(x,y)\equiv 0\bmod{p^n}
$$ 
are parametrized as 
$$
M:=\bigcup\limits_{s=0}^n M_s,
$$
where
$$
M_s:=\left\{\left(\tilde{x}\left(\frac{t}{p^s}\right),\tilde{y}\left(\frac{t}{p^s}\right)\right):t=1,..,p^{n-s},at^2+btp^s+cp^{2s}\not\equiv0\bmod{p}\right\} \mbox{ for } s=0,1,2...,n
$$
and $\tilde{x}(t),\tilde{y}(t)$ are defined as below,
\begin{equation}\label{parr}
\begin{split}
\tilde{x}(t)&:=\alpha-t\frac{At+B}{at^2+bt+c},\\
\tilde{y}(t)&:=\beta-\frac{At+B}{at^2+bt+c},
\end{split}
\end{equation}
with $A=q_x(\alpha,\beta)$ and $B=q_y(\alpha,\beta).$
\end{Proposition}
\begin{proof}
By Proposition \ref{para}, the $\mathbb{Q}_p$-rational points $(x,y)$ on the conic $q(x,y)$ are parametrized as  
 in \eqref{parr} and $q(\alpha,\beta)=0$ (in particular, $q(\alpha,\beta)=0 \bmod{p^n}$). This equation is soluble in $(\alpha,\beta)$ as a consequence of Proposition \ref{nofroot}.
Noting that $\tilde{x}(t/p^s)$ and $\tilde{y}(t/p^s)$ are $p$-adic integers, we will view $\tilde{x}(t/p^s)$ and $\tilde{y}(t/p^s)$ as elements of $\mathbb{Z}/p^n\mathbb{Z}$. 

It can be seen that the pairs $\left(\tilde{x}(t/p^s)\right),\tilde{y}(t/p^s)$ in the above sets $M_s$ ($s=0,...,n$) are distinct and $M_i\cap M_j=\emptyset$ for $i\neq j$ by the following argument.  
Suppose that
$$
\left(\tilde{x}\left(\frac{t_1}{p^{s_1}}\right),\tilde{y}\left(\frac{t_1}{p^{s_1}}\right)\right)= \left(\tilde{x}\left(\frac{t_2}{p^{s_2}}\right),\tilde{y}\left(\frac{t_2}{p^{s_2}}\right)\right),
$$
which implies 
\begin{equation*}
\begin{split}
\alpha-t_1\frac{At_1+Bp^{s_1}}{at_1^2+bt_1p^{s_1}+cp^{2s_1}}=&\alpha-t_2\frac{At_2+Bp^{s_2}}{at_2^2+bt_2p^{s_2}+cp^{2s_2}},\\
\beta-p^{s_1}\frac{At_1+Bp^{s_1}}{at_1^2+bt_1p^{s_1}+cp^{2s_1}}=&\beta-p^{s_2}\frac{At_2+Bp^{s_2}}{at_2^2+bt_2p^{s_2}+cp^{2s_2}}.
\end{split}
\end{equation*}
Then a short calculation gives
\begin{equation} \label{ok} t_2p^{s_1}\equiv t_1p^{s_2}\bmod{p^n}.
\end{equation}
So if $0\leq s_1=s_2\leq n$ then $p^{s_1}(t_2-t_1)\equiv 0\bmod{p^n}$, which implies $t_2-t_1\equiv0\bmod{p^{n-s_1}}$. Hence $t_2=t_1$ because $1\leq t_1,t_2\leq p^{n-s_1}$.
If $s_2>s_1$ then we have $t_2\equiv 0 \bmod {p^{s_2-s_1}}$, which contradicts the fact that $t_2\not\equiv 0\bmod{p}$.
Now since $a\not\equiv 0\bmod{p}$, the number of solutions of $at^2+bt+c\equiv0\bmod{p}$ in the range $1\leq t\leq p^n$ is $\left(1+\left(\frac{b^2-4ac}{p}\right)\right)p^{n-1}.$\\
Therefore, we get that $|M_0|=\left(p-1-\left(\frac{b^2-4ac}{p}\right)\right)p^{n-1},|M_n|=1$, $|M_s|=p^{n-s}-p^{n-s-1}$ for $s=1,2...n-1$.
Thus
\begin{equation*}
    \begin{split}
        |M|&=\left|\bigcup_{s=0}^nM_s\right|\\
        &=\left(p-1-\left(\frac{b^2-4ac}{p}\right)\right)p^{n-1}+(p^{n-1}-p^{n-2})+....+(p^2-p)+(p-1)+1=\left(p-\left(\frac{b^2-4ac}{p}\right)\right)p^{n-1}.
    \end{split}
\end{equation*}
Now Proposition \ref{nofroot} tells us that $M=\bigcup_{s=0}^nM_n$ is a complete set of solutions of $q(x,y)\equiv 0$ mod $p^n$.\\
\end{proof}
\begin{Proposition}[Poisson summation formula] \label{Poisson} Let $\Phi : \mathbb{R}\rightarrow \mathbb{R}$ be a Schwartz class function, $\hat\Phi$ its Fourier transform. Then
$$
\sum\limits_{n\in \mathbb{Z}} \Phi(n)=\sum\limits_{n\in \mathbb{Z}} \hat\Phi(n).
$$
\end{Proposition}

\begin{proof}
See \cite[section 3, chapter 5]{StSa}. 
\end{proof}

\begin{Proposition}[Evaluation of exponential sums with rational functions] \label{Expsums}
Let $p>2$ be a prime, $n\ge 2$ be a natural number and $f=F_1/F_2$ be a rational function where $F_1,F_2\in \mathbb{Z}[x]$. For a polynomial $G$ over $\mathbb{Z}$, let $\mbox{ord}_p(G)$ be the largest power of $p$ dividing all of the coefficients of $G$, and for a rational function $g=G_1/G_2$ with $G_1$ and $G_2$ polynomials over $\mathbb{Z}$, let $\mbox{ord}_p(g) := \mbox{ord}_p(G_1)-\mbox{ord}_p(G_2)$. Set
$$
r:=\mbox{ord}_p(f'),
$$
and 
$$
S_{\alpha}(f;p^n):=\sum\limits_{\substack{x=1\\ x\equiv \alpha \bmod{p}}}^{p^n} e_{p^n}(f(x)), 
$$
where $\alpha\in \mathbb{Z}$.
Then we have the following if $r\le n-2$ and $(F_2(\alpha),p)=1$.\medskip\\
(i) If $p^{-r}f'(\alpha)\not\equiv 0\bmod{p}$, then $S_{\alpha}(f,p^n) = 0$.\medskip\\
(ii) If $\alpha$ is a root of $p^{-r}f'(x)\equiv 0\bmod{p}$ of multiplicity one, then
$$
S_{\alpha}(f;p^n) =\begin{cases} e_{p^n}\left(f(\alpha^{\ast})\right)p^{(n+r)/2} & \mbox{ if } n-r \mbox{ is even,}\\
e_{p^n}\left(f(\alpha^{\ast})\right)p^{(n+r)/2}\left(\frac{A(\alpha)}{p}\right)\cdot \frac{G_p}{\sqrt{p}} & \mbox{ if } n-r \mbox{ is odd,}
\end{cases}
$$
where $\alpha^{\ast}$ is the unique lifting of $\alpha$ to a solution of the congruence $p^{-r}f'(x) \equiv 0 \bmod p^{[(n-r+1)/2]}$ and 
$$
A(\alpha):=2p^{-r}f''(\alpha^{\ast}).
$$
\end{Proposition}

\begin{proof} This is \cite[Theorem 3.1(iii)]{CoZ}.
\end{proof}
\begin{Proposition}\label{Hbb}
    Let $Q\in\mathbb{Z}[x,y,z]$ be a non-singular quadratic form with the associated matrix $M$. Let $\Delta_Q=|\text{det} M|$ and write $\delta_Q$ for the the highest common factor of the $2\times 2$ minors of $M$. Then
    $$
    \#\left\{(x,y,z)\in\mathbb{Z}^3:Q(x,y,z)=0,~\rm{gcd}(x,y,z)=1,~\rm{max}\{|x|,|y|,|z|\}\leq B\right\}\ll \tau(|\Delta_Q|)\left(1+\frac{B\delta_Q^{\frac{1}{2}}}{|\Delta_Q|^{\frac{1}{3}}}\right)
    .$$
\end{Proposition}
\begin{proof} This is \cite[corollary 2]{Hb}.
\end{proof}
\section{Proof of Theorem \ref{mainresult}}
\subsection{Double Poisson summation}\label{fdob} We start by writing 
\begin{equation*}
\begin{split}
T= & \sum\limits_{\substack{(x,y,z)\in \mathbb{Z}^3\\ (z,p)=1\\ Q(x,y,z)\equiv 0 \bmod{p^n}}} \Phi\left(\frac{x-x_0}{N}\right)\Phi\left(\frac{y-y_0}{N}\right)\Phi\left(\frac{z-z_0}{N}\right)\\
=& \sum\limits_{(z,p)=1} \Phi\left(\frac{z-z_0}{N}\right)\sum\limits_{\substack{\tilde{x},\tilde{y}\bmod{p^n}\\ q(\tilde{x},\tilde{y})\equiv 0\bmod{p^n}}} \sum\limits_{\substack{x\equiv \tilde{x}z\bmod{p^n}\\ y\equiv \tilde{y}z\bmod{p^n}}} \Phi\left(\frac{x-x_0}{N}\right)\Phi\left(\frac{y-y_0}{N}\right).
\end{split}
\end{equation*}
Now we apply Poisson summation, Proposition \ref{Poisson}, after a linear change of variables to the inner double sum over $x$ and $y$, obtaining
\begin{equation*}
T= \frac{N^2}{p^{2n}}\sum\limits_{(z,p)=1} \Phi\left(\frac{z-z_0}{N}\right)\sum\limits_{(k_1,k_2)\in \mathbb{Z}^2} \hat{\Phi}\left(\frac{k_1N}{p^n}\right)\hat{\Phi}\left(\frac{k_2N}{p^n}\right)e_{p^n}\left(-k_1x_0-k_2y_0\right)\sum\limits_{\substack{\tilde{x},\tilde{y}\bmod{p^n}\\q(\tilde{x},\tilde{y})\equiv 0\bmod{p^n}}} e_{p^n}\left(k_1z\tilde{x}+k_2z\tilde{y}\right).
\end{equation*}
Using the parametrization in Proposition \ref{Par}, we deduce that
\begin{equation*}
T= \frac{N^2}{p^{2n}}\sum\limits_{(z,p)=1} \Phi\left(\frac{z-z_0}{N}\right)\sum\limits_{(k_1,k_2)\in \mathbb{Z}^2} \hat{\Phi}\left(\frac{k_1N}{p^n}\right)\hat{\Phi}\left(\frac{k_2N}{p^n}\right)e_{p^n}\left(-k_1x_0-k_2y_0\right)\sum_{s=0}^n\sum\limits_{\substack{(\tilde{x},\tilde{y})\in M_s}} e_{p^n}\left(k_1z\tilde{x}+k_2z\tilde{y}\right).
\end{equation*}
We decompose $T$ into 
\begin{equation} \label{divide}
T=T_0+U,
\end{equation}
where $T_0$ is the main term contribution of $(k_1,k_2)=(0,0)$. Hence,
\begin{equation*}
\begin{split}
T_0&= \hat{\Phi}(0)^2\cdot \frac{N^2}{p^{2n}}\sum\limits_{(z,p)=1} \Phi\left(\frac{z-z_0}{N}\right) \cdot\left|M\right|\\
 &=\hat{\Phi}(0)^2\cdot \frac{N^2}{p^{2n}}\sum\limits_{(z,p)=1} \Phi\left(\frac{z-z_0}{N}\right) \cdot p^{n-1}(p-s_p(Q)),
\end{split}
\end{equation*}
where $s_p(Q)$ is defined as in \eqref{sdef} and $M$ is defined as in Proposition \ref{Par}. \\
If $N\ge p^{n\varepsilon}$ for any fixed $\varepsilon>0$, then the term $T_0$ can be simplified as
\begin{equation} \label{T0I}
\begin{split}
T_0= & \hat{\Phi}(0)^2 \cdot \frac{p-s_p(Q)}{p}\cdot \frac{N^2}{p^n} \cdot \left(\sum\limits_{z} \Phi\left(\frac{z-z_0}{N}\right) -\sum\limits_{z} \Phi\left(\frac{pz-z_0}{N}\right)\right)\\
= & \hat{\Phi}(0)^2 \cdot \frac{p-s_p(Q)}{p}\cdot \frac{N^2}{p^n}\cdot \left(N\cdot \frac{p-1}{p}\cdot \hat\Phi(0)
+\sum\limits_{w\in \mathbb{Z}\setminus\{0\}} \left(N\hat\Phi(Nw)\cdot e(-wz_0)-\frac{N}{p}\cdot \hat\Phi\left(\frac{Nw}{p}\right)\cdot e_p(-wz_0)\right)\right)\\
= & \hat{\Phi}(0)^3 \cdot \frac{(p-s_p(Q))(p-1)}{p^2}\cdot \frac{N^3}{p^{n}}\cdot\left(1+o(1)\right)=\hat{\Phi}(0)^3 \cdot C_p(Q)\cdot \frac{N^3}{p^{n}}\cdot\left(1+o(1)\right)
\end{split}
\end{equation}
as $n\rightarrow\infty$, 
where we again use Poisson summation for the sums over $z$ above and the rapid decay of $\hat\Phi$. 

\subsection{Evaluation of exponential sums}
Now we look at the error contribution
\begin{equation} \label{2errorcont}
U= \frac{N^2}{p^{2n}}\sum\limits_{(z,p)=1} \Phi\left(\frac{z-z_0}{N}\right)\sum\limits_{(k_1,k_2)\in \mathbb{Z}^2\setminus \{(0,0)\}} \hat{\Phi}\left(\frac{k_1N}{p^n}\right)\hat{\Phi}\left(\frac{k_2N}{p^n}\right)e_{p^n}\left(-k_1x_0-k_2y_0\right)\cdot E\left(k_1,k_2,z;p^n\right)
\end{equation}
with 
\begin{equation*}
\begin{split}
E\left(k_1,k_2,z;p^n\right):=\sum_{s=0}^n\sum\limits_{(\tilde{x},\tilde{y})\in M_s} e_{p^n}\left(z(k_1\tilde{x}+k_2\tilde{y})\right).
\end{split}
\end{equation*}
Assume that
$$
(k_1,k_2,p^n)=p^r.
$$
Set
$$
l_1:=\frac{k_1}{p^r}, \quad l_2:=\frac{k_2}{p^r}.
$$
The contribution of $r=n-1,n$ to the right-hand side of \eqref{2errorcont} is $O_{\varepsilon}(1)$ if $N\ge p^{n\varepsilon}$ by the rapid decay of $\hat\Phi$ since $(k_1,k_2)=(0,0)$ is excluded from the summation. In the following, we assume that $r\le n-2$ so that Proposition \ref{Expsums} is applicable.

Let $$
f_{s,k_1,k_2}(t)=z\left(k_1\tilde{x}_s(t)+k_2\tilde{y}_s(t)\right),
$$ where $\tilde{x}_s(t)=\tilde{x}\left(\frac{t}{p^s}\right),\tilde{y}_s(t)=\tilde{y}\left(\frac{t}{p^s}\right)$ and $\tilde{x}(t),\tilde{y}(t)$ is defined in  \eqref{parr}.
We have $f_{s,k_1,k_2}(t)\equiv f_{s,k_1,k_2}(t+wp^{n-s})\bmod{p^n}$ for $w=1,2...,p^{s}$.\\ 
So we deduce 
\begin{equation}\label{newsum}
\begin{split}
E\left(k_1,k_2,z;p^n\right):=\sum_{s=0}^n\frac{1}{p^s}\sum\limits_{\substack{t=1\\(at^2+btp^s+cp^{2s})\not\equiv 0~~\text{mod}~~p}}^{p^n} e_{p^n}\left( f_{s,k_1,k_2}(t)\right).
\end{split}
\end{equation}
The derivative of the amplitude function turns out to be
\begin{equation*}
\begin{split}
f'_{s,k_1,k_2}(t)&=z\frac{k_1\left((aB-bA)t^2p^{2s}-2Actp^{3s}-Bcp^{4s})\right)+k_2\left(aAt^2p^{2s}+2aBtp^{3s}+(Bb-Ac)p^{4s}\right)}{(at^2+btp^s+cp^{2s})^2 }\cdot\frac{1}{p^s}\\
&=zp^s\frac{\left(k_1(aB-bA)+k_2aA\right)t^2+2(aBk_2-cAk_1)tp^{s}+\left((bB-cA)k_2-cBk_1\right)p^{2s}}{(at^2+btp^s+cp^{2s})^2}.
\end{split}
\end{equation*}
For $ (aB-bA)l_1+aAl_2\neq 0,$ we define $$
r'=r'(l_1,l_2):=\frac{\ln(\mbox{ord}_p((aB-bA)l_1+aAl_2))}{\ln{p}}
.$$
Note that 
$$
aB^2-bAB+cA^2=-4\Delta.
$$
So 
$$cA((aB-bA)k_1+aAk_2)+(aB-bA)(aBk_2-cAk_1)=-4a\Delta k_2,$$
and 
$$B((aB-bA)k_1+aAk_2)-A(aBk_2-cAk_1)=-4\Delta k_1.$$
Thus, it can be shown that $\left(k_1,k_2,p^n\right)=\left(\left((aB-bA)k_1+aAk_2)\right),(aBk_2-cAk_1),p^n\right)$ as $ a\Delta\not\equiv0\bmod{p}.$\\Therefore,
$$\mbox{ord}_p(f_{s,k_1,k_2}')=
\begin{cases}
 p^{2s+r} & \text{if}~~  r'\geq s, \\
 p^{s+r+r'} & \text{if} ~~ r'<s.
\end{cases}$$
We split the second sum over $t$ on the right-hand side of \eqref{newsum} into
\begin{equation*}
\begin{split}
\sum\limits_{\substack{t=1\\(at^2+btp^s+cp^{2s})\not\equiv 0~~\text{mod}~~p}}^{p^n} e_{p^n}\left( f_{s,k_1,k_2}(t)\right)=\sum\limits_{\substack{\alpha=1\\(ap^{2s}+b\alpha p^s+c\alpha^2)\not\equiv 0~~\text{mod}~~p}}^p S_{\alpha}\left(f_{s,k_1,k_2};p^n\right),
\end{split}
\end{equation*}
where $$
S_{\alpha}(f_{s,k_1,k_2};p^n):=\sum\limits_{\substack{t=1\\ t\equiv \alpha \bmod{p}}}^{p^n} e_{p^n}(f_{s,k_1,k_2}(t)). 
$$
We see that if $r'=0$ and $s>0$ then $\mbox{ord}_p(f_{s,k_1,k_2}')^{-1}\cdot f'_{s,k_1,k_2}(t)\equiv 0\bmod{p}$ implies $t\equiv0\bmod{p}$. Then using Proposition \ref{Expsums}, we have $S_{\alpha}=0$ if $\alpha\neq0$, which implies that 
\begin{equation*}
 \sum_{s=1}^n\frac{1}{p^s}\sum\limits_{\substack{t=1\\(at^2+btp^s+cp^{2s})\not\equiv 0~~\text{mod}~~p}}^{p^n} e_{p^n}\left( f_{s,k_1,k_2}(t)\right)=0.
\end{equation*}
If $r'\neq 0$ and $s\neq r'$ then $\mbox{ord}_p(f_{s,k_1,k_2}')^{-1}\cdot f'_{s,k_1,k_2}(t)\equiv 0\bmod{p}$ implies $t\equiv0\bmod{p}$. Then again using Proposition \ref{Expsums}, we have $S_{\alpha}=0$ if $\alpha\neq0$, which implies that 
\begin{equation*}
 \sum\limits_{\substack{s=0\\s\neq r'}}^n\frac{1}{p^s}\sum\limits_{\substack{t=1\\(at^2+btp^s+cp^{2s})\not\equiv 0~~\text{mod}~~p}}^{p^n} e_{p^n}\left( f_{s,k_1,k_2}(t)\right)=0.
\end{equation*}
Therefore
\begin{align}\label{E}
E(k_1,k_2,z,p^n)=
\begin{cases}
    \sum\limits_{\substack{t=1\\(at^2+bt+c)\not\equiv 0~~\text{mod}~~p}}^{p^n} e_{p^n}\left( f_{0,k_1,k_2}(t)\right) & \text{if } r'=0 ,\\
    \sum\limits_{\substack{t=1\\(at^2+bt+c)\not\equiv 0~~\text{mod}~~p}}^{p^n} e_{p^n}\left( f_{0,k_1,k_2}(t)\right)+E(r') \frac{1}{p^{r'}}\sum\limits_{\substack{t=1\\(at^2+btp^{r'}+cp^{2r'})\not\equiv 0~~\text{mod}~~p}}^{p^n} e_{p^n}\left( f_{r',k_1,k_2}(t)\right)                                                               &  \text{if } r'\neq 0,
\end{cases}
\end{align}
where $$E(r')=\begin{cases}
 1 & \text{if}~~  r'\leq n, \\
 0 & \text{if} ~~ r'>n.
\end{cases}$$
So, we may write $$
U=U_1+U_2+U_3+O_{\varepsilon}(1),
$$ where 
\begin{equation}
    \begin{split}
        U_1&= \frac{N^2}{p^{2n}}\sum\limits_{(z,p)=1} \Phi\left(\frac{z-z_0}{N}\right)\sum\limits_{r=0}^{n-2}\sum\limits_{\substack{(l_1,l_2)\in \mathbb{Z}^2\\ (l_1,l_2,p)=1\\r'(l_1,l_2)= 0 }} \hat{\Phi}\left(\frac{l_1N}{p^{n-r}}\right)\hat{\Phi}\left(\frac{l_2N}{p^{n-r}}\right)e_{p^{n-r}}\left(-l_1x_0-l_2y_0\right) E\left(p^rl_1,p^rl_2,z;p^n\right), \\
        U_2&=\frac{N^2}{p^{2n}}\sum\limits_{(z,p)=1} \Phi\left(\frac{z-z_0}{N}\right)\sum\limits_{r=0}^{n-2}\sum\limits_{\substack{(l_1,l_2)\in \mathbb{Z}^2\\ (l_1,l_2,p)=1\\ (aB-bA)l_1+aAl_2\neq 0\\r'(l_1,l_2)\neq 0}} \hat{\Phi}\left(\frac{l_1N}{p^{n-r}}\right)\hat{\Phi}\left(\frac{l_2N}{p^{n-r}}\right)e_{p^{n-r}}\left(-l_1x_0-l_2y_0\right) E\left(p^rl_1,p^rl_2,z;p^n\right),\\
        U_3&=\frac{N^2}{p^{2n}}\sum\limits_{(z,p)=1} \Phi\left(\frac{z-z_0}{N}\right)\sum\limits_{r=0}^{n-2}\sum\limits_{\substack{(l_1,l_2)\in \mathbb{Z}^2\\ (l_1,l_2,p)=1\\ (aB-bA)l_1+aAl_2= 0}} \hat{\Phi}\left(\frac{l_1N}{p^{n-r}}\right)\hat{\Phi}\left(\frac{l_2N}{p^{n-r}}\right)e_{p^{n-r}}\left(-l_1x_0-l_2y_0\right)E\left(p^rl_1,p^rl_2,z;p^n\right).
    \end{split}
\end{equation}
Set 
$$
D:=(al_2^2-bl_1l_2+cl_1^2)(aB^2-bAB+cA^2) \bmod{p^n},
$$ 
and
 $$
 C_s(t):=\left(l_1(aB-bA)+l_2aA\right)t^2+2(aBl_2-cAl_1)tp^{s}+\left((bB-cA)l_2-cBl_1\right)p^{2s}.
 $$

For $r'=0$, we consider the congruence
\begin{equation}\label{2keycong}
 C_0(\alpha)\equiv0\bmod{p^{n-r}.}
 \end{equation}
 If $D$ is a quadratic non-residue modulo $p$, then the congruence \eqref{2keycong} has no solution. So by Proposition \ref{Expsums}, $E\left(p^rl_1,p^rl_2,z;p^n\right)=0.$
 Assuming $D$ to be quadratic residue modulo $p$, we calculate the roots of  the congruence \eqref{2keycong} as
$$
\frac{-(l_2Ba-l_1Ac)\pm\sqrt{D}}{(aB-bA)l_1+aAl_2}\bmod{p^{n-r}},
$$
where $\sqrt{D}$ denotes one of the two roots of the congruence
$$
x^2\equiv D\bmod{p^{n-r}}.
$$
For simplicity, we  fix one root 
$$\alpha^\ast=\frac{-(l_2Ba-l_1Ac)+\sqrt{D}}{(aB-bA)l_1+aAl_2}\bmod{p^{n-r}}.$$
So, we have 
\begin{equation}\label{Den}
    a(\alpha^{\ast})^2+b\alpha^{\ast}+c=\frac{2aD-\sqrt{D}[l_2(2a^2B-abA)+l_1(b^2A-abB-2acA)]}{[(aB-bA)l_1+aAl_2]^2}.
\end{equation}
The congruence \eqref{2keycong} has a double root $\alpha \bmod{p}$ iff $D\equiv 0 \bmod{p}$, and in this case, we get $a\alpha^2+b\alpha+c\equiv 0\bmod{p}$ because of \eqref{Den} and such $\alpha$ is already excluded from the summation. Hence, only the case $D\not\equiv 0\bmod{p}$ occurs in which we have no root if $D$ is a quadratic non-residue modulo $p$ and two roots of multiplicity one if $D$ is a quadratic residue modulo $p$. Therefore,  we may assume from now on that $D\not\equiv 0 \bmod{p}$ and $D$ is a quadratic residue modulo $p$. Then using Proposition \ref{Expsums}, if $\alpha$ satisfies \eqref{2keycong}, we obtain
$$
S_{\alpha}\left(f_{p^rl_1,p^rl_2},p^n\right)=
\begin{cases}
e_{p^n}\left(f_{0,p^rl_1,p^rl_2}(\alpha^{\ast})\right)\cdot p^{(n+r)/2} & \mbox{ if } n-r \mbox{ is even,}\\
e_{p^n}\left(f_{0,p^rl_1,p^rl_2}(\alpha^{\ast})\right)\cdot \left(\frac{A(\alpha)}{p}\right)\cdot \frac{G_p}{\sqrt{p}}\cdot p^{(n+r)/2} & \mbox{ if } n-r \mbox{ is odd,}
\end{cases}
$$
where
$$
A(\alpha)=\frac{2 f_{0,p^rl_1,p^rl_2}''(\alpha)}{p^r}.
$$
Set
$$
J_{l_1,l_2}:=l_1(be-2cd)+l_2(bd-2ae).
$$
Then, a short calculation gives
$$
e_{p^n}\left(f_{0,p^rl_1,p^rl_2}(\alpha^{\ast})\right)=e_{p^{n-r}}\left( z\cdot\overline{4ac-b^2}(J_{l_1,l_2}-2\sqrt{D})\right).
$$
Further, we calculate 
$$ 
f_{0,p^rl_1,p^rl_2}''(\alpha^{\ast})=z\cdot \frac{\sqrt{D}}{(a+b\alpha^{\ast}+c(\alpha^{\ast})^2)^2}\cdot p^{r},
$$
Therefore, 
$$
\left(\frac{A(\alpha)}{p}\right)=\left(\frac{2z\sqrt{D}}{p}\right). 
$$Similarly, for the other root of the congruence \eqref{2keycong}, we get
$$
e_{p^n}\left(f_{0,p^rl_1,p^rl_2}(\alpha^{\ast})\right)=e_{p^{n-r}}\left( z\cdot\overline{4ac-b^2}(J_{l_1,l_2}+2\sqrt{D})\right),
$$
and
$$
\left(\frac{A(\alpha)}{p}\right)=\left(\frac{-2z\sqrt{D}}{p}\right). 
$$
So altogether, we obtain
\begin{equation*} \label{2Uaftereva}
\begin{split}
U_1&= \frac{N^2}{p^{3n/2}}\sum\limits_{r=0}^{n-2} p^{r/2} \sum\limits_{\substack{(l_1,l_2,p)=1\\r'(l_1,l_2)=0\\ D=\Box\bmod{p}}} \hat{\Phi}\left(\frac{l_1N}{p^{n-r}}\right)\hat{\Phi}\left(\frac{l_2N}{p^{n-r}}\right)e_{p^{n-r}}\left(-l_1x_0-l_2y_0\right)\sum\limits_{(z,p)=1} \Phi\left(\frac{z-z_0}{N}\right)C_{n-r}(z,D)\times \\& 
\left(e_{p^{n-r}}\left(z\cdot\overline{4ac-b^2}(J_{l_1,l_2}+2\sqrt{D})\right)+\left(\frac{-1}{p^{n-r}}\right)e_{p^{n-r}}\left(z\cdot\overline{4ac-b^2}(J_{l_1,l_2}-2\sqrt{D})\right)\right)+O_{\varepsilon}(1),
\end{split}
\end{equation*}
where $D=\Box\bmod{p}$ means that $D$ is a quadratic residue modulo $p$ and 
$$
C_{n-r}(z,D):=\begin{cases} 1 & \mbox{ if } n-r \mbox{ is even,}\\ 
\left(\frac{-2z\sqrt{D}}{p}\right)\cdot \frac{G_p}{\sqrt{p}} & \mbox { if } n-r
\mbox{ is odd.} \end{cases}
$$

Now we calculate $E\left(p^rl_1,p^rl_2,z;p^n\right)$ when $r'(l_1,l_2)>0$ from \eqref{E}.
In this case, $\mbox{ord}_p(f'_{0,k_1,k_2})=r$ and $p^{-r}f'_{0,k_1,k_2}(t)\equiv 0\bmod{p}$ implies 
$$
\left((bB-cA)l_2-cBl_1\right)+2(aBl_2-cAl_1)t\equiv 0\bmod{p}.
$$
Therefore, the congruence 
\begin{equation}\label{21keycong}
 C_0(\alpha_1)\equiv 0\bmod{p^{n-r}}
 \end{equation}
 has exactly one root which is given by 
 $$
 \alpha_1^{\ast}=\frac{\frac{-(l_2Ba-l_1Ac)+\sqrt{D}}{p^{r'}}}{\frac{(aB-bA)l_1+aAl_2}{p^{r'}}}\bmod{p^{n-r}},
 $$
 where $\sqrt{D}$ is the root of the congruence 
 $$
x^2\equiv D\bmod{p^{n-r+r'}},
$$
 which satisfies $\sqrt{D}\equiv (l_2Ba-l_1Ac)\bmod{p^{r'}}.$ We now observe that 
 $$
e_{p^n}\left(f_{0,p^rl_1,p^rl_2}(\alpha_1^{\ast})\right)=e_{p^{n-r}}\left( z\cdot\overline{4ac-b^2}(J_{l_1,l_2}-2\sqrt{D})\right),
$$
and
$$
\left(\frac{A(\alpha_1)}{p}\right)=\left(\frac{2z\sqrt{D}}{p}\right).
$$
 We see that $\mbox{ord}_p(f'_{r',k_1,k_2})=r+2r'$ and $p^{-r-2r'}f'_{r'
 ,k_1,k_2}(t)\equiv 0\bmod{p}$ implies 
$$
2(aBl_2-cAl_1)t+\frac{\left((bB-cA)k_2-cBk_1\right)}{p^{r'}}t^2\equiv 0\bmod{p}.
$$
Therefore, the congruence 
\begin{equation}\label{211keycong}
 p^{-2r'}C_{r'}(\alpha_2)\equiv 0\bmod{p^{n-r}}
 \end{equation}
 has two roots, one of them being $\alpha=0,$ which has been excluded from the summation, and the other root is given by 
 $$
 \alpha_2^{\ast}=\frac{-(l_2Ba-l_1Ac)-\sqrt{D}}{\frac{l_2(Bb-Ac)-l_1Bc}{p^{r'}}}\bmod{p^{n-r}.}
 $$
 We now calculate that 
                $$
e_{p^n}\left(f_{r',p^rl_1,p^rl_2}(\alpha_2^{\ast})\right)=e_{p^{n-r}}\left( z\cdot\overline{4ac-b^2}(J_{l_1,l_2}+2\sqrt{D})\right).
$$
We further calculate that 
$$ 
f_{r',p^rl_1,p^rl_2}''(\alpha_2^{\ast})=-z\cdot \frac{\sqrt{D}}{(a+b\alpha^{\ast}+c(\alpha^{\ast})^2)^2}\cdot p^{r+2r'}.
$$
Therefore, 
$$
\left(\frac{A(\alpha_2)}{p}\right)=\left(\frac{-2z\sqrt{D}}{p}\right). 
 $$
So by \eqref{E} if $r'(l_1,l_2)\neq 0,$
\begin{equation}
\begin{split}
    E\left(p^rl_1,p^rl_2,z;p^n\right)&=C_{n-r}(z,D)\times\\ &\left(e_{p-r}(z\cdot\overline{4ac-b^2}(J_{l_1,l_2}-2\sqrt{D}))p^{\frac{(n+r)}{2}}+E(r')\frac{1}{p^{r'}}\left(\frac{-1}{p^{n-r}}\right)e_{p^{n-r}}(z\cdot\overline{4ac-b^2}(J_{l_1,l_2}+2\sqrt{D}))p^{\frac{(n+r+2r')}{2}}\right)\\&=C_{n-r}(z,D)\times \\&\left(e_{p-r}(z\cdot\overline{4ac-b^2}(J_{l_1,l_2}-2\sqrt{D}))+E(r')\left(\frac{-1}{p^{n-r}}\right)e_{p^{n-r}}(z\cdot\overline{4ac-b^2}(J_{l_1,l_2}+2\sqrt{D}))\right)p^{\frac{(n+r)}{2}}.
    \end{split}
    \end{equation}   
    Therefore, $U_1$ and $U_2$ are of almost same form. If $(zD,p)=1$, using \eqref{Fd} we write $C_{n-r}(z,D)$ more compactly as 
$$
C_{n-r}(z,D)= \frac{G_{p^{n-r}}}{p^{(n-r)/2}}\cdot 
\left(\frac{-2z\sqrt{D}}{p^{n-r}}\right).
$$
Now, using Proposition \ref{Expsums} and \eqref{E},  $U_3$ can be written as
\begin{equation*}
    \begin{split}
    U_3=&\frac{N^2}{p^{3n/2}}\sum_{(z,p)=1}\phi\left(\frac{z-z_0}{N}\right)\sum_{r=0}^{n-2}\sum\limits_{\substack{(l,p)=1\\      
 l\in\mathbb{Z} }}\hat{\Phi}\left(\frac{aAlN}{p^{n-r}}\right)\hat{\Phi}\left(\frac{(bA-aB)lN}{p^{n-r}}\right)e_{p^{n-r}}\left(-aAlx_0-(bA-aB)ly_0\right)p^{r/2}\times\\
 &e_{p^{n-r}}\left(l\cdot\frac{z(-16a\Delta+(4ac-b^2)(2af-Ad+ad\alpha+ae\beta))}{4ac-b^2}\right)\cdot\frac{G_{p^{n-r}}}{p^{(n-r)/2}}\cdot 
\left(\frac{az\Delta}{p^{n-r}}\right).
 \end{split}
\end{equation*}
So, trivially
\begin{equation*}
\begin{split}
U_3&\ll\frac{N^2}{p^{3n/2}}\cdot N\cdot\sum_{r=0}^{n-2}\frac{p^{n-r}}{N}\cdot p^{n\varepsilon}\cdot p^{r/2}\\
&\ll \frac{N^2}{p^{n/2}}\cdot p^{n\varepsilon}.
\end{split}    
\end{equation*}
 
 Thus, one can get 
\begin{equation} \label{2compact}
\begin{split}
U\leq & \frac{N^2}{p^{3n/2}}\sum\limits_{r=0}^{n-2} p^{r/2} \cdot \frac{G_{p^{n-r}}}{p^{(n-r)/2}} \cdot \sum\limits_{\substack{D=1\\ D\equiv \Box \bmod{p}}}^{\infty} F_{n-r}(D)  \sum\limits_{z\in \mathbb{Z}} \Phi\left(\frac{z-z_0}{N}\right)\cdot \left(\frac{z}{p^{n-r}}\right) \times \\& 
\left(e_{p^{n-r}}\left(z\cdot\overline{4ac-b^2}(J_{l_1,l_2}+2\sqrt{D})\right)+ \left(\frac{-1}{p^{n-r}}\right)e_{p^{n-r}}\left(z\cdot\overline{4ac-b^2}(J_{l_1,l_2}-2\sqrt{D})\right)\right)+O_{\varepsilon}\left(\frac{N^2}{p^{n(1/2-\varepsilon)}}\right),
\end{split}
\end{equation}
where 
\begin{equation} \label{2FD}
F_{n-r}(D):= \left(\frac{-2\sqrt{D}}{p^{n-r}}\right)\cdot \sum\limits_{\substack{(l_1,l_2,p)=1\\      
 (aB-bA)l_1+aAl_2\neq 0\\ -4\Delta(al_2^2-bl_1l_2+cl_1^2)=D }} \hat{\Phi}\left(\frac{l_1N}{p^{n-r}}\right)\hat{\Phi}\left(\frac{l_2N}{p^{n-r}}\right)e_{p^{n-r}}\left(-l_1x_0-l_2y_0\right).
\end{equation} 
\subsection{Single Poisson summation and final count} 
Now we split the sum over $z$ in \eqref{2compact} into sub-sums over residue classes modulo $p$ and perform Poisson summation, getting
\begin{equation*}
\begin{split}
\sum\limits_{z\in \mathbb{Z}} \Phi\left(\frac{z-z_0}{N}\right)\cdot &\left(\frac{z}{p^{n-r}}\right) \cdot e_{p^{n-r}}\left(z\cdot\overline{4ac-b^2}(J_{l_1,l_2}\pm2\sqrt{D})\right) \\=& 
 \sum\limits_{u=1}^{p} \left(\frac{u}{p^{n-r}}\right)
\sum\limits_{z\equiv u\bmod{p}} \Phi\left(\frac{z-z_0}{N}\right)\cdot
e_{p^{n-r}}\left(z\cdot\overline{4ac-b^2}(J_{l_1,l_2}\pm2\sqrt{D})\right)\\
=&  \frac{N}{p} \sum\limits_{v\in \mathbb{Z}} \left(\sum\limits_{u=1}^{p} \left(\frac{u}{p^{n-r}}\right) \cdot e_{p}\left(uv-z_0\left(v-\frac{\overline{4ac-b^2}(J_{l_1,l_2}\pm2\sqrt{D})}{p^{n-r-1}}\right)\right)\right) 
\times \\& \hat\Phi\left(\frac{N}{p}\left(v-\frac{\overline{4ac-b^2}(J_{l_1,l_2}\pm2\sqrt{D})}{p^{n-r-1}}\right)\right).
\end{split}
\end{equation*}
Using the rapid decay of $\hat\Phi$, the above is $O(N)$ if $||\frac{\overline{4ac-b^2}(J_{l_1,l_2}\pm2\sqrt{D})}{p^{n-r-1}}||\le p^{1+n\varepsilon}N^{-1}$ and negligible otherwise, when $n$ is sufficiently large. We may therefore constraint $\overline{4ac-b^2}(J_{l_1,l_2}\pm2\sqrt{D})$ to the range $[0, p^{n-r})$ and then write $\overline{4ac-b^2}(J_{l_1,l_2}\pm2\sqrt{D})=wp^{n-r-1}+l_3$, where $w=0,...,p-1$ and $|l_3|\le L_r$ with
$$
L_r:=p^{n-r+n\varepsilon}N^{-1}.
$$ 
Moreover, the summations over $l_1$ and $l_2$ in \eqref{2FD} can be cut off at $|l_1|,|l_2|\le L_r$ at the cost of a negligible error if $n$ is large enough. It follows that
\begin{equation*}
\begin{split}
U\ll & \frac{N^3}{p^{3n/2}}\sum\limits_{r=0}^{n-2} p^{r/2} \sum\limits_{w=0}^{p-1} \sum\limits_{\substack{(l_1,l_2,l_3)\in \mathbb{Z}^3\\ |l_1|,|l_2|,|l_3|\le L_r\\ \left((b^2-4ac)(wp^{n-r-1}+l_3)-J_{l_1,l_2}\right)^2\equiv 4D\bmod{p^{n-r}}}} 1 +O_{\varepsilon}\left(\frac{N^2}{p^{n(1/2-\varepsilon)}}\right).
\end{split}
\end{equation*}
Now the congruence above implies 
\begin{equation} \label{acongru}
Ml_1^2+Nl_1l_2+Ol_2^2+Pl_1l_3+Ql_2l_3+Rl_3^2\equiv 0\bmod{p^{n-r-1}},
\end{equation}
where $M=16a\Delta+(be-2cd)^2$, $N=-16b\Delta+2(be-2cd)(bd-2ae)$, $O=16c\Delta+(bd-2ae)^2$, $P=2(4ac-b^2)(2cd-be)$, $Q=2(4ac-b^2)(2ae-bd)$ and $R=(4ac-b^2)^2.$\\\\
Let $H:=\max\{|M|,|N|,|O|,|P|,|Q|,|Q|\}$.
Now if $6HL_r^2< p^{n-r-1}$, then the congruence above can be replaced by the equation 
$$
Q'(l_1,l_2,l_3):=Ml_1^2+Nl_1l_2+Ol_2^2+Pl_1l_3+Ql_2l_3+Rl_3^2= 0.
$$
Clearly, this is the case when $N\geq p^{n/2+2n\varepsilon}$.
We calculate the determinant of the associated matrix of $Q'$ as
$$
\Delta_{Q'}=64\Delta^2(4ac-b^2)^3
.$$
Therefore by our assumptions, $Q'$ is non-singular. So we use Proposition \ref{Hbb},
obtaining
\begin{equation}
\begin{split}
\sum\limits_{\substack{(l_1,l_2,l_3)\in \mathbb{Z}^3\\ |l_1|,|l_2|,|l_3|\le L_r\\Q'(l_1,l_2,l_3)=0 }} 1&\ll\sum_{d=1}^{L_r}\tau(|\Delta_{Q'}|)\left(1+\frac{L_r\delta_{Q'}^{1/2}}{|\Delta_{Q'}|^{1/3}}\cdot\frac{1}{d}\right)\\
&\ll_{\varepsilon} \tau(|\Delta_{Q'}|)\left(L_r+\frac{L_r^{1+\varepsilon}\delta_{Q'}^{1/2}}{|\Delta_{Q'}|^{1/3}}\right)\\
&\ll_{\varepsilon} L_r^{1+\varepsilon}.
\end{split}    
\end{equation}
 Thus, in the case when $N\geq p^{n/2+2n\varepsilon}$ we get 
\begin{equation}
\begin{split}
U&\ll\frac{N^3}{p^{3n/2}}\sum_{r=0}^{n-2}p^{r/2}\sum\limits_{\substack{(l_1,l_2,l_3)\in\mathbb{Z}^3\\|l_1|,|l_2|,|l_3|\le L_r\\Q'(l_1,l_2,l_3)=0}}1+O_{\varepsilon}\left(\frac{N^2}{p^{n(1/2-\epsilon)}}\right)\\&\ll \frac{N^3}{p^{3n/2}}\sum_{r=0}^{n-2}p^{r/2}\cdot L_r^{1+\varepsilon}+O_{\varepsilon}\left(\frac{N^2}{p^{n(1/2-\epsilon)}}\right)\\&\ll\frac{N^3}{p^{3n/2}}\sum_{r=0}^{n-2}p^{r/2}\cdot p^{n-r+2n\varepsilon-r\varepsilon+n\varepsilon^2}N^{-1-\varepsilon}+O_{\varepsilon}\left(\frac{N^2}{p^{n(1/2-\epsilon)}}\right)\\&\ll\frac{N^{2-\varepsilon}}{p^{n/2}}\cdot p^{n(2\varepsilon+\varepsilon^2)}+O_{\varepsilon}\left(\frac{N^2}{p^{n(1/2-\epsilon)}}\right).
\end{split}    
\end{equation}
This needs to be compared to the main term which is of size
$$
T_0\asymp \frac{N^3}{p^n}.
$$
If $N\ge p^{(1/2+2\varepsilon+\varepsilon^2)n}$, then $U=o(T_0)$, which completes the proof of Theorem \ref{mainresult} upon changing $2\varepsilon+\varepsilon^2$ into $\varepsilon$. 
\section{Some corrections to previous work \cite{arx}}
In this section, I will give some corrections to our work \cite{arx}. None of them will affect the results in \cite{arx}.
\begin{itemize}
    \item In \cite[Preliminaries]{arx} (page 5) it is said that if $q$ is odd, then 
$$
G_q=\sum\limits_{y=1}^q \left(\frac{y}{q}\right)e_q(y).
$$ 
This is only true when $q$ is an odd prime.
\item In \cite[Proposition 1]{arx} the map 
\begin{equation*}
m : \left\{t\in K : t^2\not=-\frac{\alpha_1}{\alpha_2}\right\} \longrightarrow \left\{(z_1,z_2)\in K^2 : \alpha_1z_1^2+\alpha_2z_2^2=-\alpha_3\right\}
\end{equation*}
may not always be bijective. To make this map bijective one needs to add one more point $\infty$ to the domain. See Proposition \ref{para}.
\item  In  \cite[ Proposition 5(ii)]{arx} the subscript $p^n$ to $e()$ is missing. The correct expression of $S_{\alpha}(f;p^n)$ is as follows:
$$
S_{\alpha}(f;p^n) =\begin{cases} e_{p^n}\left(f(\alpha^{\ast})\right)p^{(n+r)/2} & \mbox{ if } n-r \mbox{ is even,}\\
e_{p^n}\left(f(\alpha^{\ast})\right)p^{(n+r)/2}\left(\frac{A(\alpha)}{p}\right)\cdot \frac{G_p}{\sqrt{p}} & \mbox{ if } n-r \mbox{ is odd.}
\end{cases}
$$
\item On page 17 the correct expression for $f'_{s,k_1,k_2}$ is
\begin{equation*}
\begin{split}
f'_{s,k_1,k_2}(t)=2x_3\cdot p^s\cdot\frac{\alpha_2(k_2\alpha_1a-k_1\alpha_2b)t^2-2\alpha_1\alpha_2(ak_1+bk_2)tp^{2s}-\alpha_1(k_2\alpha_1a-k_1\alpha_2b)p^{2s}}{(\alpha_1p^{2s}+\alpha_2t^2)^2}
\end{split}
\end{equation*}
and therefore $ord_p(f'_{s,k_1,k_2})=s+r.$
\item On the same page it is also said that when $s>0$ then $p^{-r-2s}f'_{s,k_1,k_2}\equiv 0\bmod{p}$ implies $t\equiv 0\bmod{p}.$ This is not correct when $\frac{(k_2\alpha_1a-k_1\alpha_2b)}{p^r}\equiv 0\bmod{p}$. Therefore 
\begin{equation*}
 \sum_{s=1}^n\frac{1}{p^s}\sum\limits_{\substack{t=1\\t\not\equiv 0~~\text{mod}~~p}}^{p^n} e_{p^n}\left( f_{s,k_1,k_2}(t)\right)\neq0.  
\end{equation*}
One can treat the case separately when $\frac{(k_2\alpha_1a-k_1\alpha_2b)}{p^r}\equiv 0\bmod{p}$ as in \eqref{E}.
\item 
 On page 19 we should have 
$$
\left(\frac{A(\alpha)}{p}\right)=\left(\pm\frac{2x_3\alpha_2\sqrt{D}}{p}\right), 
$$
and this will change the expression for $U$ accordingly.
\end{itemize}

 All of the above-mentioned errors have been addressed in this paper.

\end{document}